\documentclass{amsart}
\usepackage{euscript,amsmath, amssymb, amsfonts, mathtools}
\usepackage{mathrsfs}
\usepackage{mathabx}
\usepackage{stmaryrd}

\numberwithin{equation}{section}

\newtheorem{theorem}{Theorem}[section]
\newtheorem{lemma}[theorem]{Lemma}

\theoremstyle{definition}

\theoremstyle{remark}
\newtheorem{remark}[theorem]{Remark}
\usepackage{graphicx}  
\usepackage{icomma}
\usepackage{bbm}
\usepackage{tikz}
\usetikzlibrary{shapes}
\usepackage{color}
\usepackage{algorithmic}
\usepackage{verbatim}
\usepackage{xurl}
\usepackage{indentfirst}
\usepackage{lipsum}
\usepackage[backend=biber]{biblatex}
\usepackage{placeins}
\usepackage{bigints}
\usepackage{subcaption}
\usepackage{appendix}
\DeclareMathOperator{\IM}{Im}
\DeclareMathOperator{\RE}{Re}
\DeclareMathOperator{\Jac}{Jac}
\usepackage{csquotes}

\begin{document}
	\title[On the solution of a conformal mapping problem]{On the solution of a conformal mapping problem by means of Weierstrass functions}
	
	\author{Smirnov Matvey}
	\address{119991 Russia, Moscow GSP-1, ul. Gubkina 8,
		Institute for Numerical Mathematics,
		Russian Academy of Sciences}
	\email{matsmir98@gmail.com}
	
	\begin{abstract}
		 The conformal mapping problem for the section of a channel filled with porous material under a rectangular dam onto the upper half-plane is considered. Similar problems arise in computing of fluid flow in hydraulic structures. As a solution method, the representation of Christoffel--Schwartz elliptic integral in terms of Weierstrass functions is used. The calculation is based on Taylor series for the sigma function, the coefficients of which are determined recursively. A simple formula for a conformal mapping is obtained, which depends on four parameters and uses the sigma function. A numerical experiment was carried out for a specific area. The degeneration of the region, which consists in the dam width tending to zero, is considered, and it is shown that the resulting formula has a limit that implements the solution of the limiting problem. A refined proof of Weierstrass recursive formula for the coefficients of Taylor series of the sigma function is presented.
		 
	\end{abstract}
	\maketitle
	
	\section{Introduction}
	A region $\Omega \subset \mathbb C$, the boundary of which is a polygonal curve with angles multiple of $\pi/2$, is considered in this article. This region models the shape of a channel under a dam. The calculation of fluid flow in such channel boils down to a conformal mapping problem of $\Omega$ onto the upper half-plane. The solution of such problems is given by Christoffel-Schwartz integral (see, e.g., \cite{ShabatEng} or \cite{SchwartzChristoffel}), which in this case is naturally defined on an elliptic Riemann surface. In this paper the simple formula, which writes integral in terms of Weierstrass sigma function (see e.g., \cite{AkhiezerEng} or \cite{Chandra}), was found. This approach allows to avoid numerical integration and the mapping parameters can be found from a simple nonlinear system of equations; therefore, the computation is significantly simplified.
	
	Similar problems have been considered in \cite{SigBog}, \cite{DamBogEng}, \cite{BogHept}, \cite{BogGrig}, and \cite{BogCond}, where Christoffel-Schwartz integral was efficiently represented by theta functions (see, e.g., \cite{FarkasKra}). In the paper \cite{BogCond} the application of Lauricella functions to these problems was studied, and also different approaches were compared.
	
	The main advantage of Weierstrass functions over theta functions is that they have limiting values when the surface degenerates. The paper analyzes behavior of the constructed conformal mapping under the condition that the dam width tends to zero. It turns out that conformal mappings have a limit, which is a solution to the limiting problem. Thus, it is shown that the solution is stable under the considered degeneration.
	
	The described property of the Weierstrass sigma function loses its value if the standard method is used for calculations, which expresses the sigma function in terms of the theta function (since it does not withstand degeneration). Thus, it is necessary to use an independent method for calculating sigma functions. In this paper, we use the expression for the coefficients of its expansion into Taylor series obtained by Weierstrass (see \cite{Weier}). Since the proof presented there is apparently not complete (at one point, the analyticity of the sigma function in three variables in a neighborhood of zero is used, which is not obvious), we give a more detailed proof in the appendix. The above formula, however, is not sufficient for the final numerical solution, since Taylor series are not suitable for calculations with large arguments (namely, such a need arises with degeneration). Thus, the problem of constructing an efficient computational method for the sigma function independent of theta functions still remains unsolved. If such a method is available, it will be possible to construct formulas that are stable under various degenerations and use them in calculations. The results of this work illustrate the need in such methods.
	
	Problems in which hyperelliptic Riemann surfaces of higher genus arise can also be solved using the theory of sigma functions developed by Klein and Baker in \cite{Klein} and \cite{Baker} respectively (more detailed exposition can be found in \cite{Buch}). There is hope that it will be possible to prove the stability of formulas expressing the solution of the above problems in terms of high-order sigma functions. Thus, the construction of Weierstrass-type recurrent formulas (which are known for genus 1 and 2; see \cite{Buch}) and calculation methods for sigma functions can be extremely useful in applied problems.
	
	\section{The statement and the origin of the problem}
	Consider region $\Omega$ in the complex plane pictured on Figure 1. 
	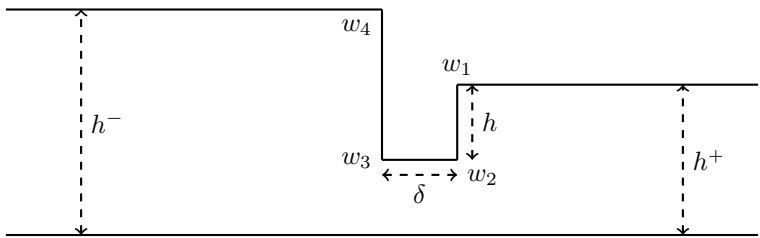
\begin{figure}[h!]
		\centering
		\begin{tikzpicture}
			\draw[black,thick] (-5,-1) -- (5,-1);
			\draw[black,thick] (-5,2) -- (0,2) node [pos=1,below left] {$w_4$};
			\draw[black,thick] (0,2) -- (0,0) node [pos=1,left] {$w_3$};
			\draw[black,thick] (0,0) -- (1,0) node [pos=1,below right] {$w_2$};
			\draw[black,thick] (1,0) -- (1,1) node [pos=1,above] {$w_1$};
			\draw[black,thick] (1,1) -- (5,1);
			\draw[<->,black,dashed,thick] (0,-0.2) -- (1,-0.2) node [pos=0.5,below] {$\delta$};
			\draw[<->,black,dashed,thick] (1.2,0) -- (1.2,1)  node [pos=0.5,right] {$h$};
			\draw[<->,black,dashed,thick] (-4,-1) -- (-4,2) node [pos=0.5,right] {$h^-$};
			\draw[<->,black,dashed,thick] (4,-1) -- (4,1) node [pos=0.5,right] {$h^+$};
		\end{tikzpicture}
		
		\caption{Region $\Omega$.}
	\end{figure}
	It is bounded from below by a line and from above by a polygonal curve with four vertices $w_1,w_2,w_3,w_4$ (it is convenient to think that this region also has two vertices at $\pm \infty$). Let the line, which bounds the region from below be parallel to the real axis, and let the vertex $w_4$ be at the origin. Then, this region is determined by four real parameters $h^-,h^+,h,\delta$, where $h$ is equal to the length of segment $[w_1,w_2]$, $\delta$ is the length of $[w_2,w_3]$, and $h^-$ and $h^+$ are equal to the distance from the line that bounds the region from below to $w_4$ and $w_1$ respectively. These parameters are positive and satisfy inequalities $h^- - h^+ + h > 0$, which corresponds to positivity of the length of $[w_3,w_4]$, and $h < h^+$. The region is determined uniquely by these parameters.

	Regions similar to $\Omega$ arise in problems connected with the computation of fluid flow through the porous material under a dam. Since the flow is continuous and satisfies the Darcy's law, the pressure $p$ is a harmonic function in $\Omega$. Assuming that segments $[w_1, w_2]$, $[w_2,w_3]$, $[w_3, w_4]$, and channel's bottom are impenetrable, we obtain natural boundary conditions: the normal derivative $\partial p/\partial n$ vanishes on impenetrable segments of the boundary, while on the remaining segments (i.e. on the half-lines starting from $w_1$ and $w_4$) $p$ is locally constant.

	Consider a real-valued function $q$ in the region $\Omega$ such that $f = p + iq$ is holomorphic (such function exists because $\Omega$ is simply connected). The normal derivative of $p$ vanishing condition is easily equivalent to constancy of $q$ on the corresponding boundary segment. It follows, that, if $f$ is a function that conformally maps $\Omega$ onto a rectangle in a way, such that $w_1$, $w_4$, and vertices at infinity are mapped to the vertices of a rectangle, then $p = \RE f$ is a solution to the original problem. $q = \IM f$ is called the current function. Its level lines are the streamlines of the fluid under the dam. It is clear, that it is enough to solve the problem of conformal mapping of $\Omega$ onto upper half-plane $\mathbb C_+$ in order to solve the specified problem.

	In what follows, the conformal mapping problem will be solved explicitly using the tools of Weierstrass elliptic functions. Below we show the calculation of streamlines in $\Omega$ obtained using the method constructed in this work.
	
	\begin{figure}[h!]
		\centering
		\begin{subfigure}{.5\textwidth}
			\centering
			\includegraphics[width=.9\textwidth]{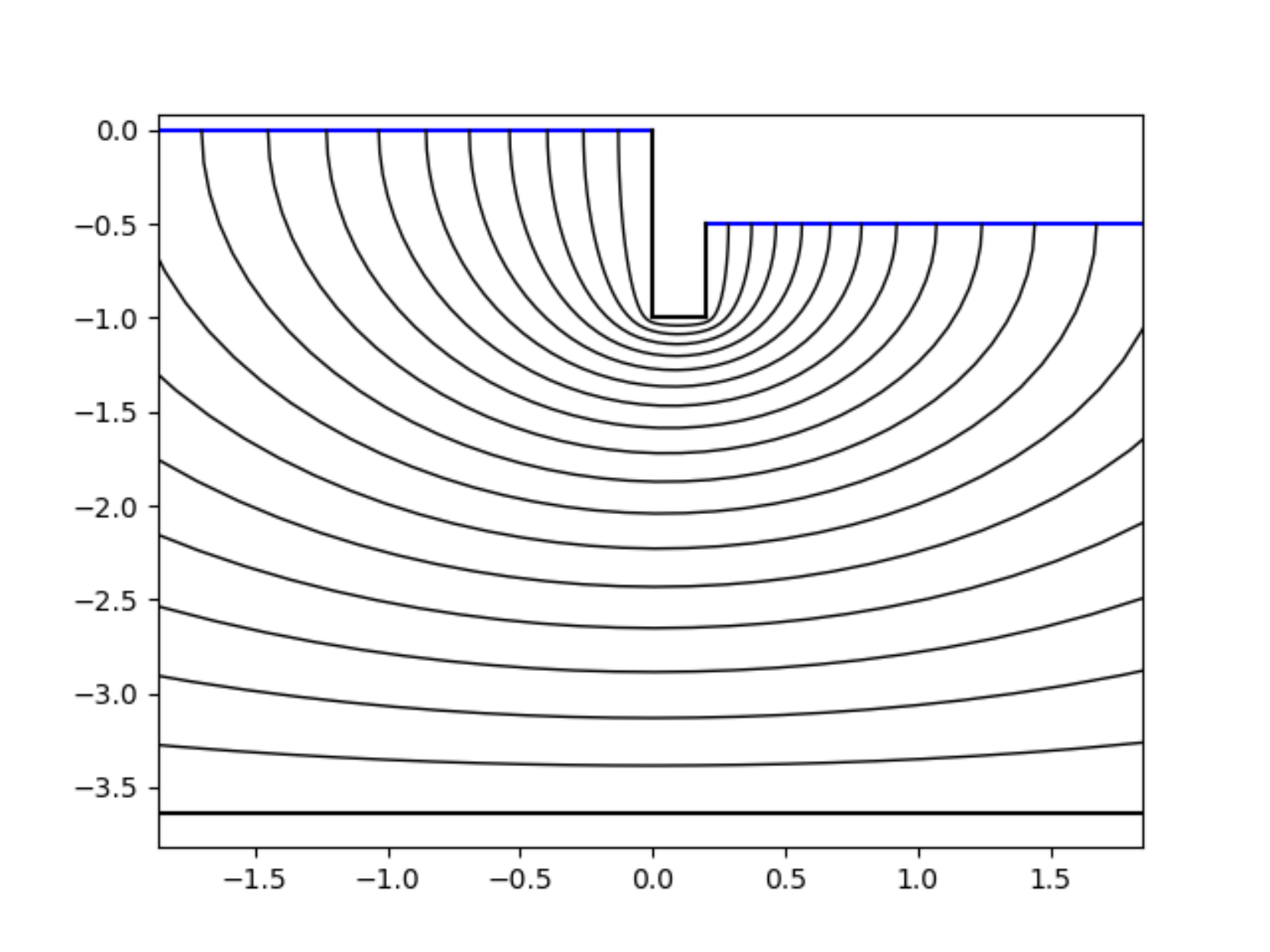}
		\end{subfigure}%
		\begin{subfigure}{.5\textwidth}
			\centering
			\includegraphics[width=.9\textwidth]{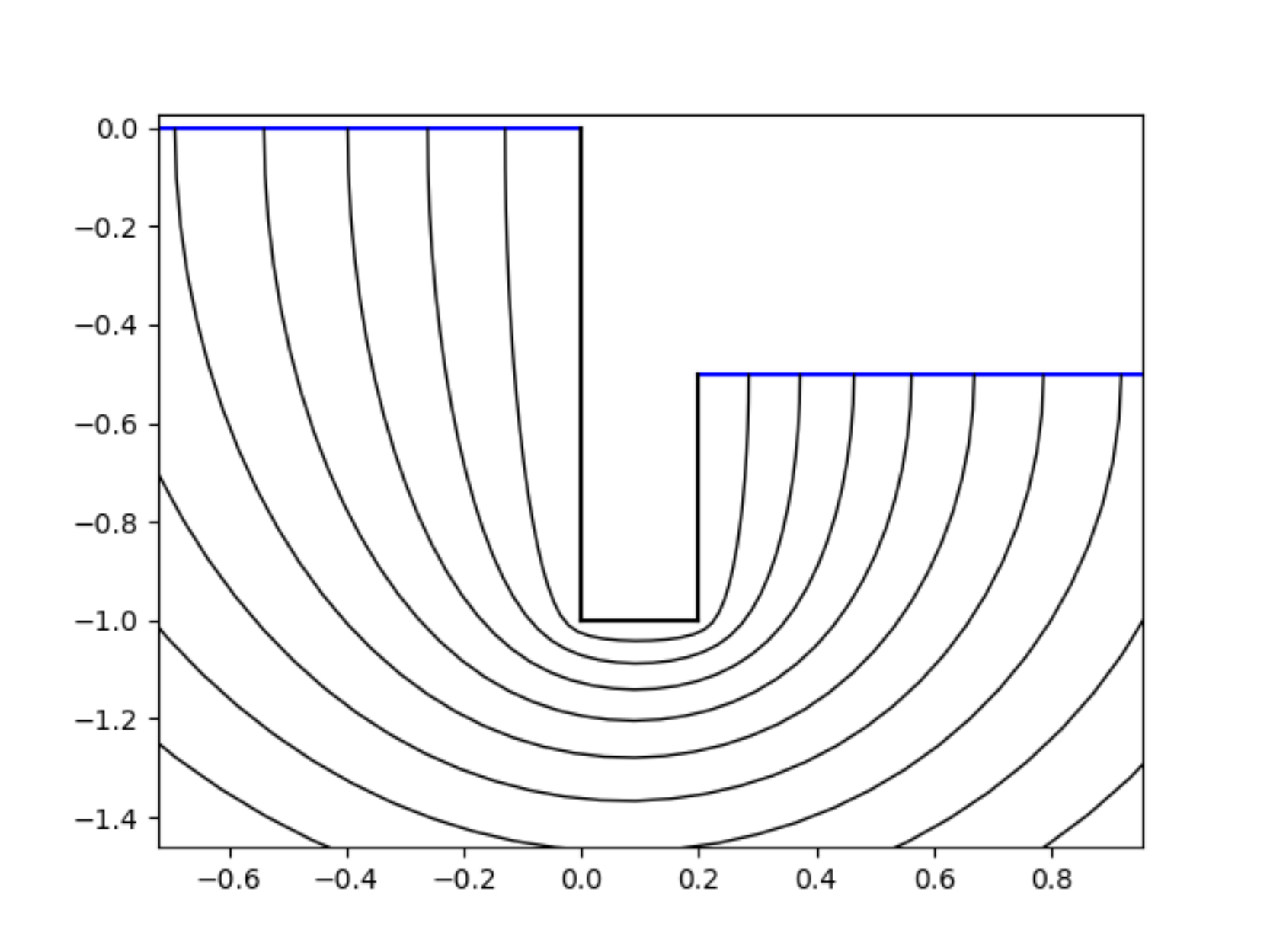}
		\end{subfigure}
		
		\caption{Streamlines in the region $\Omega$.}
	\end{figure}
	\FloatBarrier
	\section{The solution of the conformal mapping problem}
	\subsection{The general form of the solution and parameter determination}
	Since $\Omega$ is simply connected, there is a conformal mapping $W:\mathbb C_+ \rightarrow \Omega$, where $\mathbb C_+ = \{z \in \mathbb C: \IM z > 0\}$ is the upper half-plane (see, e.g., \cite{KartanEng} or \cite{ShabatEng}). Using, if necessary, a suitable automorphism of $\mathbb C_+$, one can make it such that the point $w_4$ is the preimage under $W$ (more precisely, under its continuation to the boundary) of $\infty$. Then, by Christoffel-Schwartz theorem (see \cite{SchwartzChristoffel}), there exist $x^-<x^+<x_1<x_2<x_3 \in \mathbb R$ and $C \in \mathbb C$ such that
	\begin{equation}
		dW = \phi = C \frac{\sqrt{(x - x_2)(x-x_3)}}{(x-x^-)(x-x^+)\sqrt{x-x_1}}dx.
	\end{equation}
	\begin{remark}
		Here $x_i$ is the preimage of $w_i$ under $W$, and $x^-$ and $x^+$ are the points on the boundary of the upper half-plane in which $W$ goes to infinity (preimages of the vertices at infinity).
	\end{remark}

	The differential form $\phi$ can be considered on the hyperelliptic Riemann surface $V$ of genus $1$, defined by equation $y^2 = F(x) = 4(x - x_1)(x-x_2)(x-x_3)$. Using the shift of the upper half-plane we can set $x_1 + x_2 + x_3 = 0$ without loss of generality. Thus, $F(x) = 4x^3 - g_2x - g_3$ for some real $g_2,g_3$ (that are determined by $x_1,x_2,x_3$). We can rewrite $\phi$ on this surface in the form
	\begin{equation}
		\phi = 2C \frac{(x-x_2)(x-x_3)}{y(x-x^-)(x-x^+)}dx.
	\end{equation}
	
	Let us fix the branch of $\sqrt{F(x)}$ in the region that is made from $\mathbb C$ by throwing off segment $[x_1,x_2]$ and half-line $[x_3,\infty]$. Let this branch have positive values as the argument tends to half-line $(x_3, \infty)$ from the upper half-plane. Recalling that $dx/y$ is a holomorphic (everywhere non-zero) form on $V$, we obtain that $\phi$ has two zeros of multiplicity $2$ at $(x_2,0)$ and $(x_3,0)$ and also four simple poles at $(x^-,\pm \sqrt{F(x^-)})$ and $(x^+,\pm\sqrt{F(x^+)})$. Note that the residues of this form at these poles are equal to $\pm h^-/\pi$ and $\mp h^+/\pi$ respectively.

	Now we shall use Abel map (see, e.g., \cite{ForsterEng}) which identifies $V$ with $\Jac(V)$ (as usual, we set the infinity as the initial point and $dx/y$ as the basis of holomorphic forms). Let us introduce the half-periods $$
	\omega = \int_{x_1}^{x_2} \frac{dx}{y},\;\; \omega' = -\int_{x_2}^{x_3} \frac{dx}{y},
	$$ and quantities $\eta = \zeta(\omega)$ and $\eta' = \zeta(\omega')$, where $\zeta$ is the Weierstrass zeta function (see \cite{AkhiezerEng}).
	It is easy to see that $\omega,\eta \in \mathbb R$ and $\omega', \eta' \in i\mathbb R$. The set of points $(x,\sqrt{F(x)})$, where $x \in \mathbb C_+$, is mapped by this map onto the rectangle with vertices $0,\omega', \omega' - \omega, -\omega$. Let us denote the images of the points $(x^-, \sqrt{F(x^-)})$ and $(x^+,\sqrt{F(x^+)})$ by $z^-$ and $z^+$ respectively (see Figure 3, where the preimages of points are indicated in the brackets).
	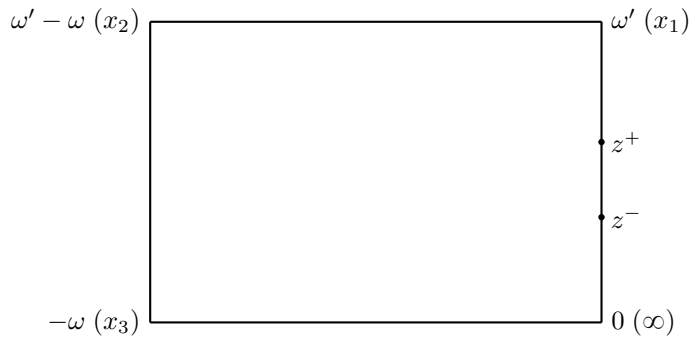
\begin{figure}[h!]
		\centering
		\begin{tikzpicture}[scale = 2]
			\draw[black,thick] (0,0) -- (0,2)node [pos=0,right] {$0\;(\infty)$};
			\draw[black,thick] (0,2) -- (-3,2)node [pos=0,right] {$\omega'\;(x_1)$};
			\draw[black,thick] (-3,2) -- (-3,0)node [pos=0,left] {$\omega' - \omega\;(x_2)$};
			\draw[black,thick] (-3,0) -- (0,0) node [pos=0,left] {$-\omega\;(x_3)$};
			\filldraw [black] (0,0.7) circle (0.5pt) node[right] {$z^-$};
			\filldraw [black] (0,1.2) circle (0.5pt) node[right] {$z^+$};
		\end{tikzpicture}
		
		\caption{Image of the upper half-plane under the Abel map.}
	\end{figure}
	The images of $(x^-, -\sqrt{F(x^-)})$ and $(x^+,-\sqrt{F(x^+)})$ in this case are equal to $-z^-$ and $-z^+$. Consider the differential form $\psi$ on the torus that corresponds to $\phi$ under this identification of $V$ with $\Jac(V)$. This form has $4$ simple poles in the points $\pm z^-$ and $\pm z^+$ and its residues are equal to $\pm h^-/\pi$ and $\mp h^+/\pi$ respectively.

	Now we use the method of representing elliptic functions by Weierstrass functions that is described in \cite{AkhiezerEng}. Consider meromorphic function
	\begin{equation}\label{eq413}
		g(z) = \frac{h^-}{\pi}(\zeta(z - z^-) - \zeta(z + z^-)) - \frac{h^+}{\pi}(\zeta(z-z^+) - \zeta(z+z^+)).
	\end{equation}
	Using the quasiperiodic properties of $\zeta$ (see \cite{AkhiezerEng}) it is easy to conclude that $g$ is elliptic. Form $g(z)dz$ has the same simple poles as $\psi$ with the same residues. Therefore, $\psi - g(z)dz$ is a holomorphic form on the torus. Since the space of holomorphic $1$-forms on the torus is one-dimensional it follows that $\psi - g(z)dz = Ddz$, where $D$ is a constant (note that $D \in i\mathbb R$).
	
	Now we return to the map $W$. It is clear that $$
	W(x) = -\int_{x}^\infty \phi.
	$$ In view of that, let \begin{equation}
		Q(z) = \int_0^z \psi.
	\end{equation}
	Obviously, $W(x)$ is equal to $Q(z)$, where $z$ is the image of $(x,\sqrt{F(x)})$ under the Abel map. Thus, $Q$ conformally maps the rectangle with vertices $0,\omega', \omega' - \omega, -\omega$ onto $\Omega$, and $\omega'$ is mapped to $w_1$, $\omega'-\omega$ is mapped to $w_2$, and $-\omega$ to $w_3$ (also $0$ is mapped to $w_4$). Now we can derive the system of equations from the previously obtained relations:
	\begin{equation}\label{eq415}
		g(-\omega) + D = 0,\;\;\; g(\omega' - \omega) + D = 0,
	\end{equation}
	\begin{equation}\label{eq416}
		Q(\omega'-\omega) - Q(\omega') = -ih,\;\;\; Q(-\omega) - Q(\omega'-\omega) = -\delta.
	\end{equation}
	\begin{remark}
		The first pair of equations follows from the fact that $\phi$ has zeros in the points $(x_2,0)$ and $(x_3,0)$, and the second pair is a consequence of relations $w_3-w_2 = -\delta$, $w_2 - w_1 = -ih$.
	\end{remark}
	It remains to derive a reasonable formula for $Q$. Recall that $\zeta$ is a logarithmic derivative of $\sigma$. It easily follows that
	\begin{equation}\label{eq417}
		Q(z) = Dz + \frac{h^-}{\pi}\ln\left(\frac{\sigma(z-z^-)}{\sigma(z+z^-)}\right) - \frac{h^+}{\pi}\ln\left(\frac{\sigma(z-z^+)}{\sigma(z+z^+)}\right) - i(h^- - h^+),
	\end{equation}
	where $\ln$ denotes the branch of logarithm in the plane cut by negative imaginary half-line such that $\ln(1) = 0$. Substituting in \eqref{eq415} the formula for $g$ from \eqref{eq413} and using quasiperiodicity of $\sigma$ (see, e.g., \cite{AkhiezerEng}), we obtain the system of equations:
\begin{equation}\label{eq418}
	\begin{dcases}
		-D\omega - \frac{2h^+}{\pi}\eta z^+ + \frac{2h^-}{\pi}\eta z^- = -ih, \\
		-D\omega' - \frac{2h^+}{\pi}\eta'z^+ + \frac{2h^-}{\pi}\eta'z^- = -\delta, \\
		D + \frac{h^-}{\pi}(\zeta(\omega - z^-) - \zeta(\omega + z^-)) - \frac{h^+}{\pi}(\zeta(\omega-z^+) - \zeta(\omega+z^+)) = 0, \\
		D + \frac{h^-}{\pi}(\zeta(\omega' + \omega - z^-) - \zeta(\omega' + \omega + z^-)) \\ \qquad-\frac{h^+}{\pi}(\zeta(\omega' + \omega-z^+) - \zeta(\omega' + \omega+z^+)) = 0.
	\end{dcases}
\end{equation}

In this system of equations there are five variables (since the quantities $\omega,\omega',\eta,\eta'$ are determined by $g_2$ and $g_3$) $g_2,g_3,D,z^+,z^-$ (the first two are real and the other are imaginary) and four equations \eqref{eq418} (the first, third and fourth equations are imaginary and the second one is real). Thus, it is natural to consider a single parameter family of curves that necessarily contains a suitable one, i.e. to consider functions $g_2 = g_2(\gamma)$ and $g_3 = g_3(\gamma)$ and use the system \eqref{eq418} to determine the parameters $\gamma, D, z^+, z^-$.

In what follows we shall use the family of curves that is defined by the roots of the polynomial $F$: $x_1 = \gamma - 1/2$, $x_2 = -2\gamma$, $x_3 = \gamma+1/2$, $\gamma \in (-1/6,1/6)$ (more detailed analysis of this family is given during the study of the degeneration $\delta \rightarrow 0$). This family corresponds to the normalization condition $x_3 - x_1 = 1$ in addition to the already given relation $x_1 + x_2 + x_3 = 0$.

\begin{figure}[ht!]
	\centering
	\begin{subfigure}{.5\textwidth}
		
		\centering
		\includegraphics[width=.9\linewidth]{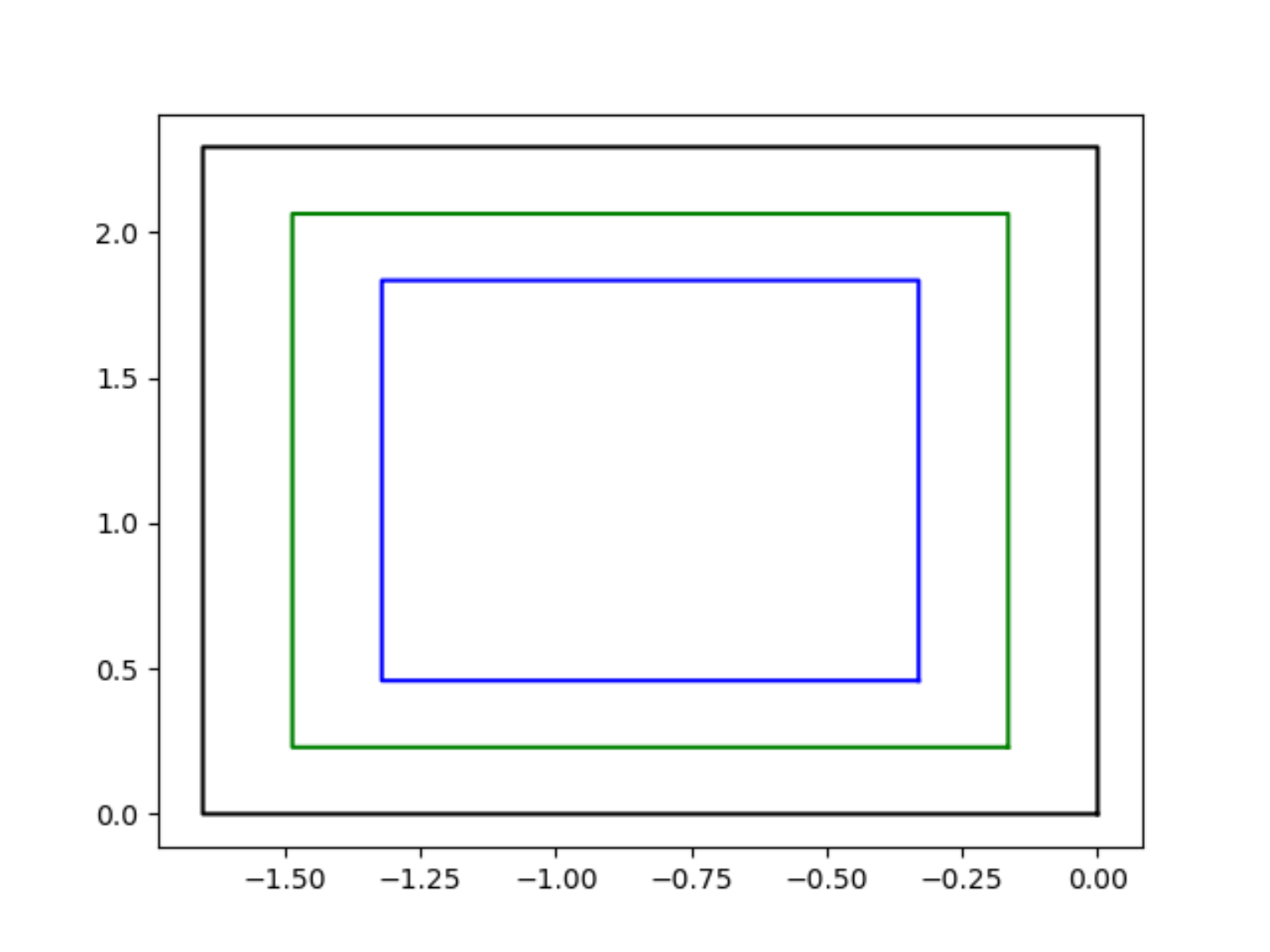}
		
		\caption{ The rectangle $P$ and contours in it.}
	\end{subfigure}%
	\begin{subfigure}{.5\textwidth}
		
		\centering
		\includegraphics[width=.9\linewidth]{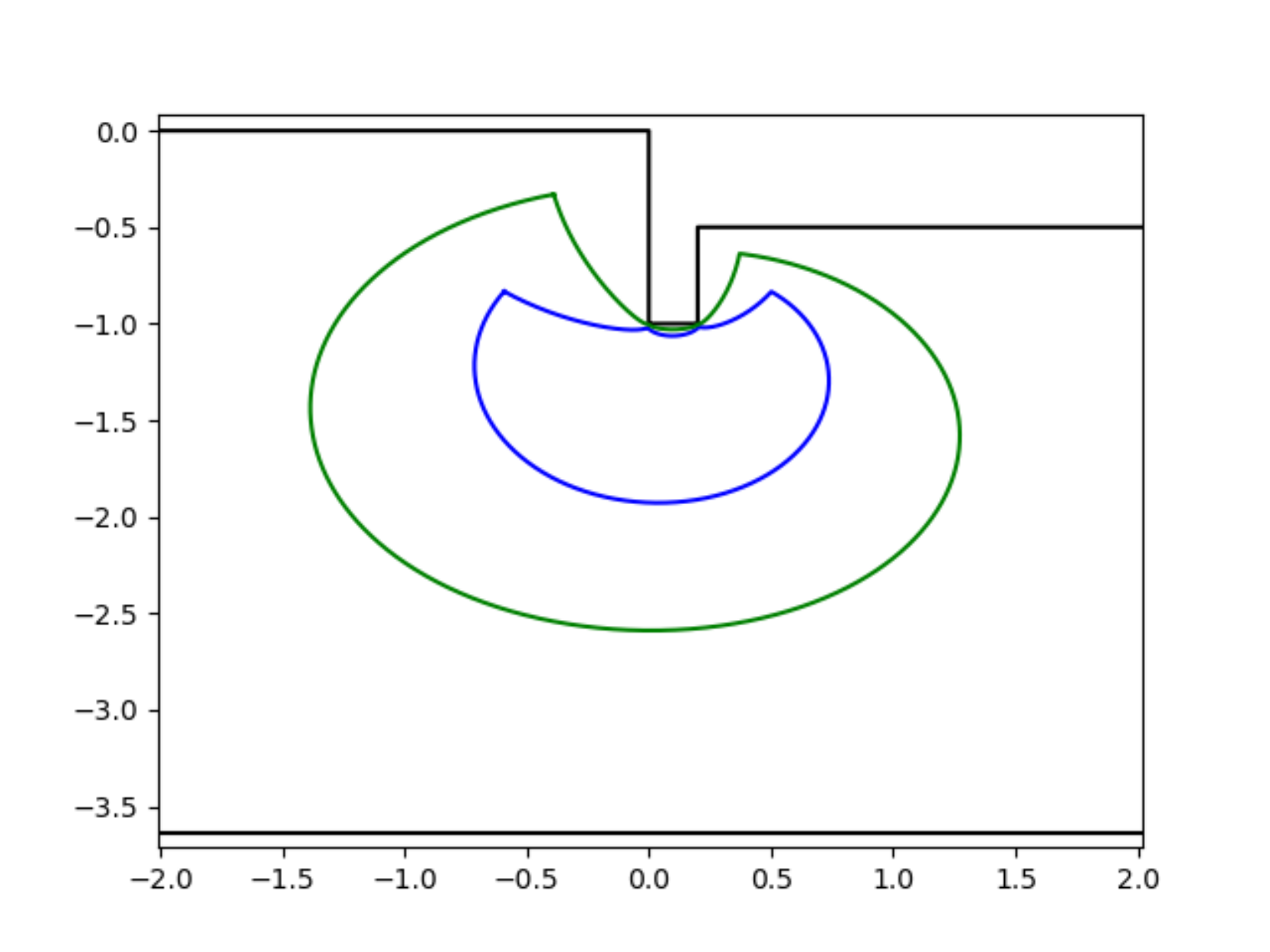}
		
		\caption{The image of the rectangle and the contours.}
	\end{subfigure}
	\begin{subfigure}{.5\textwidth}
		
		\centering
		\includegraphics[width=.9\linewidth]{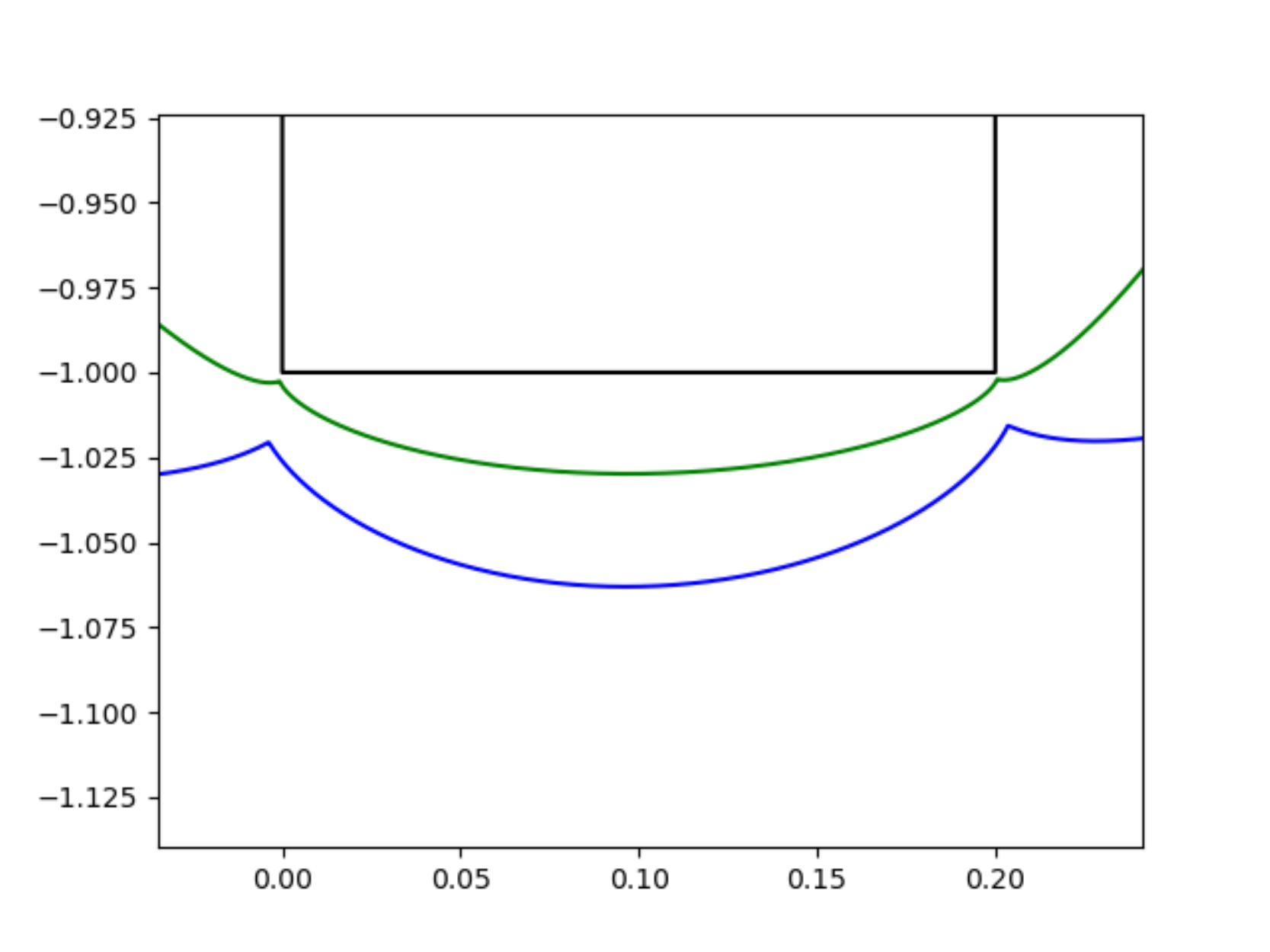}
		
		\caption{The behaviour near $[w_2,w_3]$.}
	\end{subfigure}%
	
	\caption{The conformal mapping $Q$.}
\end{figure}

\subsection{On the numerical implementation}

It was decided to use for numerical implementation the explicit computation of the sigma function depending on parameters $g_2,g_3$ through its Taylor series (see \cite{Weier} or Theorem~\ref{tWeierSeries}). It is clear that for the effective solution of the system \eqref{eq418} it is necessary to compute all the quantities in it and their derivatives with respect to parameters. In the end it reduces to the computation of $\omega$ and $\omega'$ and their derivatives with respect to $g_2$ and $g_3$ and, also, $\zeta$ and its derivatives with respect to $z, g_2,g_3$. Since $$
\zeta = \frac{1}{\sigma} \frac{\partial \sigma}{\partial z} ,
$$
the problem of computation of $\zeta$ and its derivatives can be solved easily. In order to compute $\omega$ we note that $\sigma$ has zeros exactly in the points of the lattice  $\{2m\omega + 2n\omega':n,m \in \mathbb Z\}$ and these zeros are simple. An effective way to localize a simple zero $z_0$ of a holomorphic function $f$ is to compute integral of $zf'(z)/2\pi if(z)$ on a contour enclosing $z_0$. To find a suitable contour it is possible to apply a variant of binary search using that $\omega \ge \pi/2$. Using the specified method either directly, or for an approximate calculation of zero and subsequent application of equation solving methods, it is easy to construct an effective and precise algorithm of computation of $\omega$ (and $\omega'$). In order to compute their derivatives it is possible to differentiate the integral of $z\sigma'(z)/\sigma(z)$ by $g_2$ or $g_3$, and to compute it explicitly by determining the residue in the zero of $\sigma$. Thus, the solution of the system \eqref{eq418} can be completely reduced to the computation of the sigma function and its derivatives with respect to $z$, $g_2$, and $g_3$.

We demonstrate the solution of a specific problem by this method. Let $h^+ = \pi$, $h^- = \pi + 0.5$, $h = 0.5$, $\delta = 0.2$. We shall search for the solution in the one parameter family of curves defined by $x_1 = \gamma - 1/2$, $x_2 = -2\gamma$, $x_3 = \gamma + 1/2$, $\gamma \in (-1/6, 1/6)$. The solution of the system \eqref{eq418} is $$(\gamma,D,z^+,z^-) = (0.1051616134,0.0203152915i,1.3043479103i,0.7195735824i).$$ Given this $\gamma$ we obtain $\omega = 1.6518996331$, $\omega' = 2.2939120295i$. On Figure 4 the image of the rectangle $P$ with vertices $0,\omega',\omega'-\omega,-\omega$ under the map $Q$ is shown.

\section{Stability of the solution under the degeneration of the region}
Here we shall consider a problem of conformal mapping of the upper half-plane onto the region $\widetilde{\Omega}$ 
that comes from $\Omega$ with degeneration $\delta\rightarrow 0$ (see Figure 5) and analyse behaviour of the solution under the condition that no other degeneration is happening (i.e. quantities $h^-,h^+,h,h^-+h-h^+,h^+-h$ have positive limits).

\begin{figure}[h!]
	\centering
	\begin{tikzpicture}
		\draw[black,thick] (-5,-1) -- (5,-1);
		\draw[black,thick] (-5,2) -- (0,2) node [pos=1,below left] {$w_4$};
		\draw[black,thick] (0,2) -- (0,0) node [pos=1,below] {$w_2= w_3$};
		\draw[black,thick] (0,1) -- (5,1) node [pos=0,left] {$w_1$};
		\draw[<->,black,dashed,thick] (0.2,0) -- (0.2,1)  node [pos=0.5,right] {$h$};
		\draw[<->,black,dashed,thick] (-4,-1) -- (-4,2) node [pos=0.5,right] {$h^-$};
		\draw[<->,black,dashed,thick] (4,-1) -- (4,1) node [pos=0.5,right] {$h^+$};
	\end{tikzpicture}
	\caption{Region $\widetilde \Omega$.}
\end{figure}
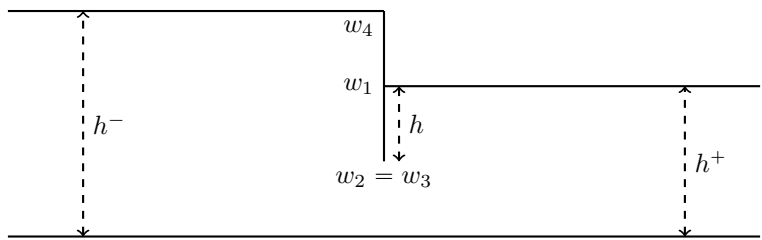
$\widetilde{\Omega}$ is determined by three parameters
$h,h^+$ and $h^-$. A conformal mapping of the upper half-plane onto $\widetilde\Omega$ can be found by the analogous method (using Christoffel-Schwartz theorem). In this case, since the corresponding Riemann surface has genus $0$, the solution can be expressed in elementary functions. Another method (that is considered here) is to apply formula \eqref{eq417}, using the fact that $\sigma$ is defined for $g_2$ and $g_3$ such that $F(x) = 4x^3 - g_2x - g_3$ has multiple roots. It is natural to suppose that the solution can be found by taking the limit under gluing the roots, that are mapped to $w_2$ and $w_3$. Together with that the stability of the solution under $\delta\rightarrow 0$ will be proved.

\subsection{Gluing of the roots}
Again consider a family of curves depending on $\gamma \in (-1/6, 1/6)$ that is given by $F_\gamma(x) = 4(x - x_1(\gamma))(x - x_2(\gamma))(x - x_3(\gamma))$, where $x_1(\gamma) = \gamma - 1/2$, $x_2(\gamma) = -2\gamma$, $x_3(\gamma) = \gamma + 1/2$. Under $\gamma \rightarrow -1/6$ the roots $x_2$ and $x_3$ glue. The limiting values of $g_2$ and $g_3$ are $4/3$ and $-8/27$ respectively. For each $\gamma$ we define quantities $\omega(\gamma)$, $\omega'(\gamma)$, $\eta(\gamma)$, $\eta'(\gamma)$. In what follows we shall omit dependence on $\gamma$.

\begin{lemma}\label{l41}
	Under $\gamma \rightarrow -1/6$ we have
	\begin{equation}\label{eq441}
		\omega,\eta \rightarrow \infty,\;\; \omega' \rightarrow \frac{i\pi}{2},\;\;\eta' \rightarrow -\frac{i\pi}{6},\;\; \frac{\eta}{\omega} \rightarrow -\frac{1}{3}.
	\end{equation}
	Moreover,
	\begin{equation}\label{eq442}
		\begin{gathered}
		\sigma(z,\frac{4}{3}, -\frac{8}{27}) = e^{-\frac{z^2}{6}} \sinh(z),\;\; \zeta(z,\frac{4}{3}, -\frac{8}{27}) = \coth(z) - \frac{z}{3},\\ \wp(z,\frac{4}{3}, -\frac{8}{27}) = \frac{1}{\sinh^2(z)} + \frac{1}{3}.
		\end{gathered}
	\end{equation}
	Finally, there exists $\varepsilon > 0$ such that for $\gamma + 1/6 < \varepsilon$ the estimation
	\begin{equation}\label{eq44asympt}
		-c_1\ln(\gamma + 1/6) \le \omega(\gamma) \le -c_2\ln(\gamma + 1/6)
	\end{equation}
	holds, where $0 < c_1 < c_2$.
\end{lemma}
\begin{proof}
	\eqref{eq441} easily follows from integral representations
	$$
	\omega(\gamma) = \frac{1}{2}\int_{\gamma - \frac{1}{2}}^{-2\gamma} \frac{dx}{\sqrt{(x - \gamma - 1/2)(x - \gamma + 1/2)(x + 2\gamma)}}, $$ $$
	\omega'(\gamma) = \frac{1}{2}\int_{-2\gamma}^{\gamma+\frac{1}{2}} \frac{dx}{\sqrt{-(x - \gamma - 1/2)(x - \gamma + 1/2)(x + 2\gamma)}}, $$ $$
	\eta(\gamma) = -\frac{1}{2}\int_{\gamma - \frac{1}{2}}^{-2\gamma} \frac{xdx}{\sqrt{(x - \gamma - 1/2)(x - \gamma + 1/2)(x + 2\gamma)}},
	$$ $$
	\eta'(\gamma) = -\frac{1}{2}\int_{-2\gamma}^{\gamma + \frac{1}{2}} \frac{xdx}{\sqrt{-(x - \gamma - 1/2)(x - \gamma + 1/2)(x + 2\gamma)}}. $$
	To derive \eqref{eq442} one can pass to the limit $\gamma \rightarrow -1/6$ using the formula that represents $\sigma$ as the infinite product (see \cite{AkhiezerEng}). We obtain
	$$
	\sigma(z,\frac{4}{3}, -\frac{8}{27}) = z\prod_{n \ne 0} \left(1 - \frac{z}{i n \pi}\right) e^{\frac{z}{i n \pi} - \frac{z^2}{2 n^2 \pi^2}}.
	$$
	Using classical identities
	$$
	\sum_{n = 1}^\infty \frac{1}{n^2} = \frac{\pi^2}{6},\;\; \prod_{n = 1}^\infty \left(1 - \frac{x^2}{n^2 \pi^2}\right) = \frac{\sin(x)}{x},
	$$
	we derive the first equation of \eqref{eq442}. The rest follow from equalities $\zeta(z) = \sigma'(z)/\sigma(z)$, $\wp(z) = -\zeta'(z)$.
	
	Now we estimate the growth of $\omega(\gamma)$. Consider equality
	$$
	\omega(\gamma) = \frac{1}{2}\int_0^{1/2 - 3\gamma} \frac{dt}{\sqrt{t(1-t)(1/2 - 3\gamma - t)}}.
	$$
	Note that for small $\gamma$ integral on the segment $[0,1/2]$ is bounded and, therefore,
	$$
	\omega(\gamma) \sim \frac{1}{2}\int_{1/2}^{1/2 - 3\gamma} \frac{dt}{\sqrt{t(1-t)(1/2 - 3\gamma - t)}}.
	$$
	The estimation of the integral in rhs is quite simple since $1/\sqrt{t}$ is bounded from below and above by positive constants and the remaining integral can be calculated explicitly.
\end{proof}

Since $\omega' \rightarrow i\pi/2$ and $\omega \rightarrow \infty$, it is natural to suppose that \eqref{eq417} can give a formula for a conformal mapping of the half-strip $$S = \{z \in \mathbb C: \RE z < 0, \IM z \in (0,\pi/2)\}$$  onto $\widetilde\Omega$. Let
\begin{equation}\label{eq443}
	\widetilde Q(z) = Dz + \frac{h^-}{\pi}\ln\left(\frac{\sigma(z-z^-)}{\sigma(z+z^-)}\right) - \frac{h^+}{\pi}\ln\left(\frac{\sigma(z-z^+)}{\sigma(z+z^+)}\right) - i(h^- - h^+),
\end{equation}
where $\sigma$ is taken at values $g_2 = 4/3$, $g_3 = -8/27$ and $D \in \mathbb C$, $z^-, z^+\in(0,i\pi/2)$ are the parameters. Substituting \eqref{eq442} into \eqref{eq443}, we obtain
\begin{equation}\label{eq444}
	\begin{multlined}
	\widetilde Q(z) = z\left(D + \frac{2h^-z^-}{3\pi} - \frac{2h^+z^+}{3\pi}\right) \\+ \frac{h^-}{\pi}\ln\frac{\sinh(z - z^-)}{\sinh(z+z^-)} - \frac{h^+}{\pi}\ln\frac{\sinh(z - z^+)}{\sinh(z + z^+)}- i(h^- - h^+).
	\end{multlined}
\end{equation}
It is easy to check that if $\widetilde Q$ has a non-zero linear term it has no limit under $\RE z \rightarrow -\infty$. If$$
D + \frac{2h^-z^-}{3\pi} - \frac{2h^+z^+}{3\pi} = 0,
$$ its limit is equal to $2(h^-z^- - h^+z^+)/\pi - i(h^- - h^+)$. Thus, if $\widetilde Q$ conformally maps $S$ onto $\widetilde \Omega$, then the conditions
\begin{equation}\label{eq445}
	\begin{dcases}
		D + \frac{2h^-z^-}{3\pi} - \frac{2h^+z^+}{3\pi} = 0, \\
		h^- z^- - h^+ z^+ = -\frac{ih\pi}{2}
	\end{dcases}
\end{equation}
hold. Obviously, these conditions are sufficient under the additional assumption that the derivative of $\widetilde{Q}$ is not vanishing on $S$ and its boundary. Using tedious but elementary calculations it can be shown that this condition is equivalent to 
\begin{equation}\label{eq446}
	h^-\sinh(2z^-) = h^+\sinh(2z^+).
\end{equation}
Thus, \eqref{eq445} and \eqref{eq446} determine the parameters $D,z^-,z^+$, at which $\widetilde Q$ is the desired conformal mapping.
 
We show that the obtained formula reduces to Christoffel-Schwartz integral in the upper half-plane under the variable change $x = \wp(z) = 1/\sinh^2(z) + 1/3$. Changing the variable in \eqref{eq443} and using that the linear term vanishes we obtain the following formula for the conformal mapping of $\mathbb C_+$ onto $\widetilde{\Omega}$:
 \begin{equation}\label{eq448}
	\widetilde{W}(x) = \frac{h^-}{\pi}\ln \left(\frac{\sqrt{x^- + 2/3} - \sqrt{x + 2/3}} {\sqrt{x^- + 2/3} + \sqrt{x + 2/3}}\right) + \frac{h^+}{\pi}\ln \left(\frac{\sqrt{x^+ + 2/3} - \sqrt{x + 2/3}} {\sqrt{x^+ + 2/3} + \sqrt{x + 2/3}}\right).
\end{equation}
 
 The equations on the parameters $z^-$ and $z^+$ can be rewritten in the form
 \begin{equation}\label{eq449}
 	\begin{dcases}
 		\frac{h^-\sqrt{x^- + 1}}{x^-} = \frac{h^+\sqrt{x^++1}}{x^+}, \\
 		h^- \ln\left(\frac{1+\sqrt{x^- + 4/3}}{\sqrt{x^- + 1/3}}\right) + h^+\ln\left(\frac{1+\sqrt{x^+ + 4/3}}{\sqrt{x^+ + 1/3}}\right) = -\frac{ih\pi}{2}.
 	\end{dcases}
 \end{equation}
Differentiating $\widetilde{W}$ and using the first equation from \eqref{eq449} we obtain  $$
d\widetilde{W} = \frac{h^-\sqrt{x^-+1}- h^+\sqrt{x^++1}}{\pi} \frac{(x - 1/3)dx}{\sqrt{x+2/3}(x-x^-)(x-x^+)}.
$$
Therefore we found the exact form of the constant in Christoffel-Schwartz integral. The remaining parameters $x^-$ and $x^+$ can be found from the system of equations \eqref{eq449}.

\subsection{Passing to the limit}

Consider a sequence of regions determined by the parameters $(h^-_n, h^+_n, h_n, \delta_n)$. Assume that they have limits $(h^-_{\lim}, h^+_{\lim}, h_{\lim}, 0)$. Also we assume that $h^+_{\lim} - h_{\lim} > 0$ and $h^-_{\lim} - h^+_{\lim} + h_{\lim} > 0$. We shall prove that the parameters $(D_n, \gamma_n, z^-_n, z^+_n)$, given by the solution of the system \eqref{eq418} for the corresponding regions, have limits $(D_{\lim}, -1/6, z^-_{\lim}, z^+_{\lim})$, and, moreover, parameters $(D_{\lim}, z^-_{\lim}, z^+_{\lim})$ satisfy \eqref{eq445} and \eqref{eq446}. Thus, in view of the fact that the Weierstrass sigma function is entire, it follows that the constructed solution is stable.

The following proof is rather long and technical and, therefore, we shall omit most of the calculations.

In the following estimations we shall also use parameters $x^-_n, x^+_n, x_1^{(n)}, x_2^{(n)}, x_3^{(n)}, C_n$ of the map $W_n$.

\begin{lemma}
	Assume that $\gamma_n \rightarrow -1/6$ and $\delta_n  \omega(\gamma_n) \rightarrow 0$.
	Then the indicated convergence holds.
\end{lemma}
\begin{proof}
	Note that $\gamma_n \rightarrow -1/6$ implies that sequences $z^-_n$ and $z^+_n$ are bounded. The first equation in \eqref{eq418} then implies that $D_n$ is also bounded. Passing to the subsequences we can assume that all these sequences are convergent (if we succeed to prove that the limits satisfy \eqref{eq445} and \eqref{eq446}, then uniqueness of the solution implies that all the subsequences converge to the same limit, and, therefore, the initial sequences converge). The first equation in \eqref{eq445} is obtained by passing to the limit in the second equation in \eqref{eq418} in view of Lemma \ref{l41}. Multiplying the first equation in \eqref{eq418} by $\omega'$ and the second one by $\omega$, subtracting and passing to limit leads to the second equation in \eqref{eq445} (the term $\delta \omega$ by assumption tends to zero). 
	
	Now we derive \eqref{eq446} for the limits of the sequences. Recall the following notation of the theory of elliptic functions (see \cite{AkhiezerEng}):
		$$
	\zeta_2(z) = \zeta(z + \omega) - \eta,
	$$
	$$
	\zeta_3(z) = \zeta(z + \omega + \omega') + \eta + \eta'.
	$$
	These functions are connected to $\sigma_2, \sigma_3$:
	$$
	\zeta_k = \frac{1}{\sigma_k} \frac{d \sigma_k}{d z} = \frac{d \ln \sigma_k}{d z}.
	$$
	Finally, $\sigma_k = \sigma \sqrt{\wp - x_k}$. The last two equations in \eqref{eq418} can be rewritten as	
	\begin{equation}\label{eq447}
		D + \frac{h^-}{\pi}(\zeta_k(-z^-) - \zeta(z^-)) - \frac{h^+}{\pi}(\zeta_k(-z^+) - \zeta(z^+)) = 0, \;\; k = 2,3.
	\end{equation}

	Since
	$$
	\begin{multlined}
		\zeta_2(z) - \zeta_3(z) = \frac{\sigma'(z) \sqrt{\wp - x_2} + \wp'(z) \sigma(z) (\wp - x_2)^{-1/2}}{\sigma(z) \sqrt{\wp - x_2}} \\ - \frac{\sigma'(z) \sqrt{\wp - x_3} + \wp'(z) \sigma(z) (\wp - x_3)^{-1/2}}{\sigma(z) \sqrt{\wp - x_3}},
	\end{multlined}
	$$
	it follows that
	\begin{equation}\label{zetadiff}
		\zeta_2(z) - \zeta_3(z) = \frac{\wp'(z)(x_2 - x_3)}{(\wp - x_2)(\wp - x_3)}.
	\end{equation}
	Equation \eqref{eq446} is derived by passing to the limit (using Lemma \ref{l41}) from equations \eqref{eq447} (from which the constant $D$ can be eliminated) and substitution the formula for $\zeta_2 - \zeta_3$ from \eqref{zetadiff}.
\end{proof}
\begin{lemma}
	Inequalities $|C_n| \ge a_1$, $|C_n| \le a_2 \sqrt{x_3^{(n)} - x^-_n}$ hold for some positive constants $a_1,a_2$. Moreover, the sequence $x^+_n$ is bounded from below.
\end{lemma}
\begin{proof}
	The estimations for $C^{(n)}$ easily follow from
	$$|C_n|\bigintss_{x_3^{(n)}}^{+\infty} \frac{\sqrt{\left(x - x_2^{(n)}\right)\left(x - x_3^{(n)}\right)}dx}{\sqrt{x - x_1^{(n)}}(x - x^-_n)(x - x^+_n)} = h^-_n - h^+_n + h_n.
	$$
	To prove the boundedness from below for $x^+_n$ it suffices to consider equality
	$$
	|C_n|\bigintss_{x_1^{(n)}}^{x_2^{(n)}} \frac{\sqrt{\left(x_2^{(n)}-x\right)\left(x_3^{(n)} - x\right)}dx}{\sqrt{x - x_1^{(n)}}(x - x^-_n)(x - x^+_n)} = h_n.
	$$
\end{proof}
\begin{lemma}
	Assume that sequence $x^-_n$ is bounded from below. Then there exist constants $0 < b_1 < b_2$ such that $b_1 \sqrt{\delta_n} \le |x_2^{(n)} - x_3^{(n)}| \le b_2 \sqrt{\delta_n}$. 
 
\end{lemma}
\begin{proof}
	It follows from easy estimation for the integral in equality
	$$
	|C_n|\bigintss_{x_2^{(n)}}^{x_3^{(n)}} \frac{\sqrt{\left(x - x_2^{(n)}\right)\left(x_3^{(n)} - x\right)}dx}{\sqrt{x - x_1^{(n)}}(x - x^-_n)(x - x^+_n)} = \delta_n.
	$$
\end{proof}

The foregoing Lemmas imply that it remains to prove that sequence $x^-_n$ is bounded from below (in view of the asymptotics \eqref{eq44asympt}). Assume that this is not true. Passing to the subsequences we can assume that $x^-_n \rightarrow -\infty$ and $x_1^{(n)},x_2^{(n)},x_3^{(n)}$, $x^+_n$ are convergent.

\begin{lemma}
	In the foregoing assumptions statements $x_3^{(n)} - x_2^{(n)} \rightarrow 0$, $x_1^{(n)} - x^+_n \rightarrow 0$ hold.
\end{lemma}
\begin{proof}
	Equality
	$$
	|C_n| \frac{\sqrt{\left(x_2^{(n)} - x^+_n\right)\left(x_3^{(n)} - x^+_n\right)}}{\sqrt{x_1^{(n)} - x^+_n}(x^+_n - x^-_n)} = \frac{h^+_n}{\pi}
	$$
	implies $x_1^{(n)} - x^+_n \rightarrow 0$. 
	
	To prove that $x_3^{(n)} - x_2^{(n)} \rightarrow 0$ we return to parameters $(D_n, \gamma_n, z^-_n, z^+_n)$. Assume that $x_2^{(n)} - x_1^{(n)} \rightarrow 0$. Then $\omega'(\gamma_n), \eta'(\gamma_n) \rightarrow \infty$, and $\omega$ and $\eta$ have finite limits. Moreover, Legendre identity (see, e.g., \cite{AkhiezerEng} or \cite{Chandra}) implies
	$$
	\lim_{n \rightarrow \infty} \frac{\omega'(\gamma_n)}{\eta'(\gamma_n)} = \lim_{n \rightarrow \infty} \frac{\omega(\gamma_n)}{\eta(\gamma_n)}.
	$$
	The first two equations from \eqref{eq418} imply
	$$
	\lim_{n \rightarrow \infty}\left(\frac{2 h^+_n}{\pi} z^+_n \left(\frac{\omega'}{\eta'} - \frac{\omega}{\eta}\right) - i \frac{h_n}{\omega}\right) = 0.
	$$
	Therefore, 
	$$
	\lim_{n \rightarrow \infty}\left(\frac{h^+_n z^+_n}{\omega'} - h_n\right) = 0,
	$$
	and, passing to the limit, we obtain $h_{\lim} \ge h^+_{\lim}$. This contradicts the assumptions made.
	
	Now assume that $x_2^{(n)} - x_1^{(n)} \nrightarrow 0$ and$x_3^{(n)} - x_2^{(n)} \nrightarrow 0$. Then both periods $\omega$ and $\omega'$ have finite limits. In this case $z^-_n \rightarrow 0$, $z^+_n \rightarrow \omega'(\gamma_{\lim})$. It is obvious that $D_n$ also is convergent and, passing to the limit in the second equation in \eqref{eq418}, we obtain
	$$
	-D_{\lim} \omega' - \frac{2h^+_{\lim} \omega' \eta'}{\pi} = 0.
	$$
	Substituting into the first equation we get
	$$
	\frac{2h^+_{\lim}}{\pi} \omega \eta' - \frac{2h^+_{\lim}}{\pi} \omega' \eta = -ih_{\lim},
	$$
	implying $h^+_{\lim} = h_{\lim}$.
\end{proof}

Now we have enough preparation to deduce a contradiction from $x^-_n \rightarrow -\infty$. In order to do this we shall analyse asymptotics of some sequences (in what follows the equivalence of sequences means that the quotient of them tends to $1$).

Equality
$$
|C_n| \frac{\sqrt{\left(x_2^{(n)} - x^-_n\right)\left(x_3^{(n)} - x^-_n\right)}}{\sqrt{x_1^{(n)} - x^-_n}(x^+_n - x^-_n)} = \frac{h^-_n}{\pi}
$$
implies that
\begin{equation}\label{asympt1}
	|C_n| \sim \frac{h^-_n}{\pi} \sqrt{|x^-_n|}.
\end{equation}
On the other hand
$$
|C_n| \frac{\sqrt{\left(x_2^{(n)} - x^+_n\right)\left(x_3^{(n)} - x^+_n\right)}}{\sqrt{x_1^{(n)} - x^+_n}(x^+_n - x^-_n)} = \frac{h^+_n}{\pi},
$$
and, therefore,
\begin{equation}\label{asympt2}
	\sqrt{x_1^{(n)} - x^+_n} \sim \frac{h^-_n}{h^+_n} \frac{1}{\sqrt{|x^-_n|}}.
\end{equation}
Now consider equality
$$
|C_n|\bigintss_{x_1^{(n)}}^{x_2^{(n)}} \frac{\sqrt{\left(x_2^{(n)} - x\right)\left(x_3^{(n)} - x\right)}dx}{\sqrt{x - x_1^{(n)}}(x - x^-_n)(x - x^+_n)} = h_n.
$$
Using \eqref{asympt1}, it is easy to show that the sequence in lhs is equivalent to sequence$$
\frac{h^-_n}{\pi \sqrt{|x^-_n|}} \bigintss_{x_1^{(n)}}^{x_2^{(n)}} \frac{\sqrt{\left(x_2^{(n)} - x\right)\left(x_3^{(n)} - x\right)}dx}{\sqrt{x - x_1^{(n)}}(x - x^+_n)}.
$$
Now, changing the variable, we obtain
$$
\bigintss_{x_1^{(n)}}^{x_2^{(n)}} \frac{\sqrt{\left(x_2^{(n)} - x\right)\left(x_3^{(n)} - x\right)}dx}{\sqrt{x - x_1^{(n)}}(x - x^+_n)} = \bigintss_{0}^{x_2^{(n)} - x_1^{(n)}} \frac{\sqrt{\left(x_2^{(n)} - x_1^{(n)} - x\right)(1 - x)}dx}{\sqrt{x}(x + x_1^{(n)} - x^+_n)}.
$$
It appears that the asymptotics of the last integral is independent of the convergence rate of $|x_2^{(n)} - x_1^{(n)}| \rightarrow 1$. Namely, for all sequences $\alpha_n \rightarrow 1$ and $a_n \rightarrow 0$ the equivalence
$$
\int_0^{\alpha_n} \frac{\sqrt{(\alpha_n - x)(1 - x)}dx}{\sqrt{x}(x + a_n)} \sim \int_0^1 \frac{(1-x)dx}{\sqrt{x}(x+a_n)} \sim \frac{\pi}{\sqrt{a_n}}
$$
holds. Finally, in view of \eqref{asympt2}, we obtain
$$
h_n = |C_n|\bigintss_{x_1^{(n)}}^{x_2^{(n)}} \frac{\sqrt{\left(x_2^{(n)} - x\right)\left(x_3^{(n)} - x\right)}dx}{\sqrt{x - x_1^{(n)}}(x - x^-_n)(x - x^+_n)} \sim \frac{h^-_n}{\pi \sqrt{|x^-_n|}} \frac{\pi}{\sqrt{x_1^{(n)} - x^+_n}} \sim h^+_n.
$$
It contradicts the assumption $h_{\lim} < h^+_{\lim}$.

\section{Conclusion}
The simple expression through the Weierstrass sigma function for a conformal mapping of a polygonal region $\Omega$ was obtained. For the specific example the numerical experiment was carried out. The behaviour under degeneration was analyzed and it was shown that the formula is stable and converges to the solution of the limiting problem.

The future research direction can be connected either with the construction and analysis of the solutions to similar problems corresponding, for example, to Riemann surfaces of genus $2$, or with the development of the sigma function theory: construction of the recurrent formulas for higher genus and elaboration of computational methods independent of the theta function theory.

\section{Acknowledgements}

The author expresses his gratitude to A. Bogatyrev and O. Grigoriev for posing the problem and useful discussions, and also to K. Malkov for help in the computer implementation of the calculations. The author also thanks the Center for Continuing Professional Education ``Sirius University'' for the invitation to the educational module ``Computational Technologies, Multidimensional Data Analysis, and Modelling'', during which some of the results of this work were obtained.

\begin{appendices}
\section{On the coefficients of the Weierstrass sigma function Taylor series}
	Here we prove that the sigma function is an entire function of three variables and derive a recurrent formula for its Taylor series coefficients, that was originally established by Weierstrass in \cite{Weier}. The proof given there has a gap connected to the analyticity of the sigma function in a neighbourhood of zero. Perhaps, this fact can be proved by an independent argument but, since Weierstrass does not give any references (and omits this issue completely), we decided to provide a complete proof here.

	The homogeneity condition
	$$
	\sigma(\frac{z}{\lambda}, \lambda^4 g_2, \lambda^6 g_3) = \frac{1}{\lambda} \sigma(z,g_2,g_3)
	$$
	easily implies the following differential equation for the $\sigma$ function:
	\begin{equation}\label{eq321}
		z \frac{\partial \sigma}{\partial z} - 4g_2\frac{\partial \sigma}{\partial g_2} - 6g_3 \frac{\partial \sigma}{\partial g_3} - \sigma = 0.
	\end{equation}
	Further, using the definition of the $\sigma$ function and the standard differential equation for the $\wp$ function, one can derive an equation (for a proof see \cite{Halphen})
	\begin{equation}\label{eq322}
		\frac{\partial^2 \sigma}{\partial z^2} - 12g_3 \frac{\partial \sigma}{\partial g_2} - \frac{2}{3} g_2^2 \frac{\partial \sigma}{\partial g_3} + \frac{1}{12} g_2 z^2 \sigma = 0.
	\end{equation}
	
	Let $f$ be an entire function of three variables $(z,g_2,g_3)$ satisfying \eqref{eq321} and \eqref{eq322}. We derive a relation between the Taylor series coefficients $f_{mnk}$ of $f$:
	$$
	f = \sum_{m,n,k = 0}^\infty f_{mnk} g_2^m g_3^n z^k.
	$$
	\eqref{eq321} implies that $f_{mnk} = 0$, if $k \ne 4m + 6n + 1$. Therefore, $f$ can be written in the form
	$$
	f = \sum_{m,n = 0}^\infty a_{mn} g_2^m g_3^n z^{4m + 6n + 1}.
	$$
	Now, substituting this expression of $f$ into \eqref{eq322}, we obtain the equality
	\begin{equation}\label{eq323}
		a_{mn} = \frac{12(m+1)a_{m+1,n-1} + \frac{2}{3}(n+1)a_{m-2,n+1} - \frac{1}{12}a_{m-1,n}}{(4m + 6n + 1)(4m + 6n)},
	\end{equation}
	in which for convenience $a_{mn}$ is defined by zero when $m$ or $n$ is negative. It is easy to see that \eqref{eq323} uniquely determines sequence $a_{mn}$ for given $a_{00}$. To prove this let us introduce an order relation on pairs of nonnegative integers $(m,n)$: $(m,n) \le (m',n')$, if $m+n < m'+n'$ or if $m+n = m'+n'$ and $n \le n'$.
	It is easy to see that we defined a well-order on $\mathbb Z_+ \times \mathbb Z_+$, and in \eqref{eq323} indices of the terms $a_{mn}$ in rhs are strictly less then $(m,n)$. Thus, it is proved that \eqref{eq323} determines $a_{mn}$ recursively for given $a_{00}$.
	
	If the sigma function was an entire function of three variables, or, at least, holomorphic in some neighbourhood of zero, then the recurrence relation \eqref{eq323} for its Taylor series coefficients would be proved. The difficulty is that the domain of $\sigma$ is the set $\{(z,g_2,g_3) \in \mathbb C^3: g_2^3 - 27 g_3^2 \ne 0\}$. The following considerations prove the entirety of $\sigma$ and the recurrence relation \eqref{eq323}.
	
	\begin{remark}
		It is known (see, e.g., \cite{AkhiezerEng} or \cite{Chandra}) that condition $g_2^3 - 27 g_3^2 \ne 0$ is equivalent to simplicity of the roots of polynomial $4x^3 - g_2 x - g_3$.
	\end{remark}
	\begin{lemma}\label{l32}
		Let $a_{mn}$ satisfy the recurrence relation \eqref{eq323}. Then for all $q > (28 + \sqrt{811})/36 \approx 1.569$ there exists $C > 0$ such that 
		\begin{equation}\label{eq324}
			|a_{mn}| \le C\frac{q^{2m + 3n}}{(2m+3n)!}.
		\end{equation}
	\end{lemma}
	\begin{proof}
		Substituting in \eqref{eq323} this estimation and it is easy to show that for the existence of a constant, it suffices that the inequality
		$$
		 \frac{6(m+1)}{4m+6n+1}\frac{q^{2m+3n-1}}{(2m+3n)!} + \frac{(n+1)q^{2m+3n-1}}{6(4m+6n+1)(2m+3n)!} + \frac{q^{2m+3n-2}}{48(2m+3n)!} \le \frac{q^{2m+3n}}{(2m+3n)!}
		$$
		holds starting from some index $(m,n)$ (in terms of the foregoing ordering). For this, in turn, it suffices to satisfy the inequality
		$$
		\frac{1}{48} + q\left(\frac{3}{2} + \frac{1}{18}\right) < q^2.
		$$
		Solving the quadratic equation, we obtain the required statement.
	\end{proof}
	
	Lemma \ref{l32} allows to define an entire function
	$$h(z,g_2,g_3) = \sum_{m,n = 0}^\infty a_{mn}g_2^m g_3^n z^{4m+6n+1},$$
	where $a_{mn}$ are determined by recurrence relation \eqref{eq323} and initial condition $a_{00} = 1$. We shall prove that $h \equiv \sigma$ for $(g_2,g_3)$ such that $g_2^3 - 27 g_3^2 \ne 0$.
	
	\begin{lemma}
		Let $f$ be a holomorphic function of variables $(z,g_2,g_3)$, defined on a set $\mathbb C \times U$, where $U \subset \mathbb C^2$ is open, satisfying equation \eqref{eq322}. Assume that $f$ is odd in variable $z$. Then $f$ can be represented by series \begin{equation}\label{eq325}
			f(z,g_2,g_3) = \sum_{n = 0}^\infty c_n(g_2,g_3) z^{2n+1},
		\end{equation}
		and in $U$ the recurrence relation
		\begin{equation}\label{eq326}
			(2n+3)(2n+2)c_{n+1} - 12g_3\frac{\partial c_n}{\partial g_2} - \frac{2}{3} g_2^2 \frac{\partial c_n}{\partial g_3} + \frac{1}{12}g_2 c_{n-1},
		\end{equation}
		holds, where $n \ge 0$ (for $n = 0$ we set $c_{n-1} = 0$).
	\end{lemma} 
	\begin{proof}
		Indeed the representability of $f$ by the series follows from entirety of $f$ by $z$. Its coefficients $c_n(g_2,g_3)$ are given by 
		$$
		c_n(g_2,g_3) = \frac{1}{(2n+1)!} \frac{\partial^{2n+1} f}{\partial z^{2n+1}}|_{z = 0}.
		$$
		It is easy to see that the series \eqref{eq325} can be differentiated term-by-term, and therefore we can substitute it in \eqref{eq322}. Collecting the coefficient at $z^{2n+1}$, we obtain \eqref{eq326}.
	\end{proof}
	
	Recurrence relation \eqref{eq326} can be used to prove, that $\sigma$ and $h$ coincide on the domain of the $\sigma$ function. Indeed, if the first terms of their expansions coincide, then these function coincide (note that they are both odd in $z$). Indeed, $\partial \sigma/\partial z |_{z = 0} \equiv \partial h/\partial z |_{z = 0} \equiv 1$. Thus, $h$ is the analytic continuation of the $\sigma$ function to an entire function of variables $(z,g_2,g_3)$. This completes the proof of the following theorem.
	\begin{theorem}[Weierstrass]\label{tWeierSeries}
		The $\sigma$ function is entire and for all $(z,g_2,g_3) \in \mathbb C^3$ equality
		\begin{equation}\label{eq327}
			\sigma(z,g_2,g_3) = \sum_{m,n = 0}^\infty a_{mn}g_2^m g_3^n z^{4m+6n+1}
		\end{equation}
		holds, where coefficients $a_{mn}$ are determined by recurrence relation \eqref{eq323} and initial condition $a_{00} = 1$.
	\end{theorem}
\end{appendices}
\printbibliography
\end{document}